\documentclass[a4paper,10pt]{amsart}
\usepackage[english]{babel}
\usepackage{amsmath,amssymb,amsthm}
\usepackage{multicol}
\usepackage[titletoc]{appendix}

\usepackage[pdftex]{color}
\usepackage[bookmarks=true,hyperindex,pdftex,colorlinks, citecolor=blue,linkcolor=blue, urlcolor=blue]{hyperref}
\usepackage[dvipsnames]{xcolor}

\usepackage{graphicx}

\usepackage[normalem]{ulem}

\theoremstyle{plain}
\newtheorem{theorem}{Theorem}

\newtheorem{proposition}[theorem]{Proposition}

\theoremstyle{definition}

\newtheorem*{definition}{Definition}

 \DeclareMathOperator{\Id}{\mathrm{Id}}
 \DeclareMathOperator{\dist}{dist\,}

\DeclareMathOperator{\Area}{Area}
\DeclareMathOperator{\NA}{NA}

\newcommand{\K}{\mathbb{K}}

\newcommand{\R}{\mathbb{R}}
\newcommand{\N}{\mathbb{N}}

\newcommand{\eps}{\varepsilon}

\renewcommand{\geq}{\geqslant}
\renewcommand{\leq}{\leqslant}

\begin{document}
	\title{There is no operatorwise version of the Bishop--Phelps--Bollob\'as property}

\dedicatory{Dedicated to the memory of Bernardo Cascales}

\author[Dantas]{Sheldon Dantas}
\address[Dantas]{Departament of Mathematics, POSTECH, 790-784, (Pohang), Republic of Korea \newline
\href{http://orcid.org/0000-0001-8117-3760}{ORCID: \texttt{0000-0001-8117-3760}  }}
\email{\texttt{sheldongil@postech.ac.kr}}

\author[Kadets]{Vladimir Kadets}
\address[Kadets]{School of Mathematics and Computer Sciences \\
V. N. Karazin Kharkiv National University \\ pl.~Svobody~4 \\
61022~Kharkiv \\ Ukraine
\newline
\href{http://orcid.org/0000-0002-5606-2679}{ORCID: \texttt{0000-0002-5606-2679} }
}
\email{\texttt{v.kateds@karazin.ua}}

\author[Kim]{Sun Kwang Kim}
\address[Kim]{Department of Mathematics, Chungbuk National University, 1 Chungdae-ro, Seowon-Gu, Cheongju,
Chungbuk 28644, Republic of Korea \newline
\href{http://orcid.org/0000-0002-9402-2002}{ORCID: \texttt{0000-0002-9402-2002}  }}
\email{\texttt{skk@chungbuk.ac.kr}}

\author[Lee]{Han Ju Lee}
\address[Lee]{Department of Mathematics Education, Dongguk University -- Seoul, 04620 (Seoul), Republic of Korea \newline
\href{http://orcid.org/0000-0001-9523-2987}{ORCID: \texttt{0000-0001-9523-2987}  }
}
\email{\texttt{hanjulee@dongguk.edu}}

\author[Mart\'{\i}n]{Miguel Mart\'{\i}n}
\address[Mart\'{\i}n]{Departamento de An\'{a}lisis Matem\'{a}tico, Facultad de
 Ciencias, Universidad de Granada, 18071 Granada, Spain \newline
\href{http://orcid.org/0000-0003-4502-798X}{ORCID: \texttt{0000-0003-4502-798X} }
 }
\email{\texttt{mmartins@ugr.es}}

\begin{abstract} Given two real Banach spaces $X$ and $Y$ with dimensions greater than one, it is shown that there is a sequence $\{T_n\}_{n\in \N}$ of norm attaining norm-one operators from $X$ to $Y$ and a point $x_0\in X$ with $\|x_0\|=1$, such that
	$$
    \|T_n(x_0)\|\longrightarrow 1 \quad \text{ but } \quad
	\inf_{n\in \N} \bigl\{\dist\bigl(x_0,\,\{x\in X\colon
    \|T_n(x)\|=\|x\|=1\}\bigr)\bigr\} >0.
    $$
This shows that a version of the Bishop--Phelps--Bollob\'{a}s property in which the operator is not changed is possible only if one of  the involved Banach spaces is one-dimensional.
\end{abstract}

\date{June 25th, 2018}

\thanks{The first author was supported by Pohang Mathematics Institute (PMI), POSTECH, Korea and Basic Science Research Program through the National Research Foundation of Korea (NRF) funded by the Ministry of Education, Science and Technology (NRF-2015R1D1A1A 09059788). The research of the second author is done in frames of Ukrainian Ministry of Science and Education Research Program 0118U002036, and it was partially supported by Spanish MINECO/FEDER projects MTM2015-65020-P and MTM2017-83262-C2-2-P. Third author was partially supported by Basic Science Research Program through the National Research Foundation of Korea(NRF) funded by the Ministry of Education, Science and Technology (NRF-2017R1C1B1002928). Fourth author was partially supported by Basic Science Research Program through the National Research Foundation of Korea (NRF) funded by the Ministry of Education, Science and Technology (NRF-2016R1D1A1B03934771). Fifth author partially supported by Spanish MINECO/FEDER grant MTM2015-65020-P}


\maketitle

\thispagestyle{plain}

\section{Introduction}
Let $X$ be a Banach space. We denote by $X^*$, $S_X$, and $B_X$ the topological dual, the unit sphere, and the closed unit ball of $X$, respectively. We say that $x^* \in X^*$ \emph{attains its norm} (or that $x^*$ is a \emph{norm attaining functional}) if there exists $x_0 \in S_X$ such that $|x^*(x_0)| = \|x^*\| = \sup_{x \in S_X} |x^*(x)|$. It is well-known that the set of all norm attaining functionals $\NA(X)$ is always norm-dense in $X^*$. This is the famous 1961 Bishop--Phelps theorem \cite{BP}. Shortly after this result was established, Bollob\'as \cite{Bol} sharped it in the following way:  given $0 < \eps < 1/2$, $x \in B_X$, and $x^* \in S_{X^*}$ satisfying that $|1 - x^*(x)| < \eps^2/2$, there are $x_0 \in S_X$ and $x_0^* \in S_{X^*}$ such that $x_0^*(x_0) = 1$, $\|x_0 - x\| < \eps$ and $\|x_0^*- x^*\| < \eps$ (we are giving the statement in a little bit improved
form, which can be found in \cite{CKG} or \cite{ChKMMR-2014}).

If $X$, $Y$ are Banach spaces, we denote by $\mathcal{L}(X, Y)$ the space of all bounded linear operators from $X$ to $Y$ and we say that $T \in \mathcal{L}(X, Y)$ \emph{attains its norm} (or that $T$ is \emph{norm attaining}) if there is $x_0 \in S_X$ such that $\|T(x_0)\| = \|T\| = \sup_{x \in S_X} \|T(x)\|$. Lindenstrauss \cite{Lind} was the first one who studied the possible validity of the Bishop--Phelps theorem for operators, i.e., the density of the set of norm attaining operators between two Banach spaces. He showed that such density is not always true and also gave some conditions on the involved Banach spaces $X$ and $Y$ to get the density of the set of norm attaining operators. We refer to the survey paper \cite{Acosta-RACSAM} for an account of the results on this area. In 2008, M.~Acosta, R.~Aron, D.~Garc\'ia, and M.~Maestre \cite{AAGM} introduced the so-called Bishop--Phelps--Bollob\'as property to check when we can get a Bollob\'as' type theorem for bounded linear operators. More precisely, a pair $(X, Y)$ of Banach spaces has the \emph{Bishop--Phelps--Bollob\'as property} (\emph{BPBp} for short) if, given $\eps > 0$, there is $\eta(\eps) > 0$ such that whenever $T \in \mathcal{L}(X, Y)$ with $\|T\| = 1$ and $x_0 \in S_X$ satisfy $\|T(x_0)\| > 1 - \eta(\eps)$, there are $S \in \mathcal{L}(X, Y)$ with $\|S\| = 1$ and $x_1 \in S_X$ such that $\|S(x_1)\| = 1$, $\|x_1 - x_0\| < \eps$, and $\|S - T\| < \eps$. Among other results, they showed that any pair of finite dimensional Banach spaces have the BPBp and characterized the pairs $(\ell_1, Y)$ to satisfy it via a geometric property on $Y$. After 2008, a lot of attention was given to this topic and there is a vast literature about the Bishop--Phelps--Bollob\'as property. We refer the reader to the very recent papers \cite{AMS,CGKS,Cho-Choi,DGMM} and references therein.  It is important to remark that the Bishop--Phelps--Bollob\'{a}s property has geometric consequences on the involved Banach spaces. For instance, if $X$ is a finite-dimensional Banach space, then all operators from $X$ into any other Banach space $Y$ attain their norm but, unless the dimension of $X$ is equal to one, it is possible to construct a renorming $\widetilde{X}$ of $X$ and to find a Banach space $Y$ such that the pair $(\widetilde{X},Y)$ fails the BPBp \cite[Theorem 3.1]{ACKLM}.

In the last years, some variations of the BPBp have appeared in the literature. For instance, there is a property, stronger than the BPBp, in which only the operator moves: a pair $(X, Y)$ of Banach spaces has the \emph{pointwise Bishop--Phelps--Bollob\'as property} \cite{DKKLM, DKL} if given $\eps>0$, there is $\eta(\eps) > 0$ such that whenever $T \in \mathcal{L}(X, Y)$ with $\|T\| = 1$ and $x_0 \in S_X$ satisfy $\|T(x_0)\| > 1 - \eta(\eps)$, there is $S \in \mathcal{L}(X, Y)$ with $\|S\| = 1$ such that $\|S(x_0)\| = 1$ and $\|S - T\| < \eps$. That is, the new operator $S$ attains its norm at the same point at which $T$ almost attains its norm. This property has deep consequences on the structure of the involved spaces as, for instance, if a pair $(X,Y)$ has the pointwise Bishop--Phelps--Bollob\'{a}s property, then $X$ has to be uniformly smooth \cite[Proposition 2.3]{DKL} (actually,  if $Y$ is equal to the base field, this characterizes uniform smoothness). If $(X,Y)$ has the pointwise Bishop--Phelps--Bollob\'{a}s property for every Banach space $Y$, then the space $X$ also has to be uniformly convex with a power type \cite[Theorem 3.1]{DKKLM}.

Thinking on an ``operatorwise'' version of the above property, the following definition appeared in \cite{D} (with the name of ``property 2''), where it is shown that many pairs of (even finite-dimensional) Banach spaces fail it.

\begin{definition}[\mbox{\cite[Definition 2.8]{D}}]
Let $X$, $Y$ be Banach spaces. The pair $(X,Y)$ has \emph{property (P2)} if given $\eps>0$, there exists $\bar\eta(\eps)>0$ such that whenever $T\in \mathcal{L}(X,Y)$ with $\|T\|=1$ and $x_0\in S_X$ satisfy that
\begin{equation*}
\|T(x_0)\|> 1 - \bar\eta(\eps),
\end{equation*}
then there is $x_1\in S_X$ such that
\begin{equation*}
	\|T(x_1)\|=1 \quad  \text{and} \quad  \|x_1 - x_0\|<\eps.
\end{equation*}
\end{definition}

For the case when $Y$ is the base field, this property had appeared earlier in \cite{KL}, where it is proved that a Banach space $X$ is uniformly convex if and only if the pair $(X,\K)$ has property (P2) \cite[Theorem 2.1]{KL}. On the other hand, it is immediate that the pairs of the form $(\K,Y)$ have property (P2) for every Banach space $Y$. Our aim in this paper is to prove that, for {\bf real} Banach spaces, these are the only possible cases in which property (P2) can be satisfied: if the real Banach spaces $X$ and $Y$ have dimension greater than or equal to two, then the pair $(X,Y)$ fails property (P2).

Let us finally comment that there is a property weaker than property (P2) also introduced in \cite{D} (with the name of property 1) where the function $\eps\longmapsto \eta(\eps)$ depends on the operator $T$. This property is satisfied, for instance, by the pairs $(\ell_p,\ell_q)$ for $1\leq q<p<\infty$ \cite{D} and it has some geometric consequences as it has been pointed out in \cite{Talponen}.

\vspace*{1cm}

We would like to dedicate this paper to the memory of our dear friend \emph{Bernardo Cascales}, who passed away last April, 2018. Bernardo was an enormous mathematician who in the last years worked, among many other topics, on the Bishop--Phelps--Bollob\'{a}s property. His deep knowledge of functional analysis, his enthusiasm, and his nice way to explain mathematics, have had an huge impact both on the BPBp and on the people working on it. We would like to highlight the following references \cite{cas-alt23, CKG, CGKS} containing his contributions to this field.

\section{The Result}

Let us state the main result of the paper.

\begin{theorem}\label{theorem:dualpropertyisnotpossible}
Let $X$ and $Y$ be real Banach spaces of dimension greater than or equal to $2$. Then the pair $(X,Y)$ fails property (P2). In other words, one may find a sequence $\{T_n\}_{n\in \N}$ of norm attaining norm-one elements of $\mathcal{L}(X,Y)$ and a point $x_0\in S_X$, such that
	$$
    \|T_n(x_0)\|\longrightarrow 1 \quad \text{ and } \quad
	\inf_{n\in \N} \bigl\{\dist\bigl(x_0,\,\{x\in S_X\colon
    \|T_n(x)\|=1\}\bigr)\bigr\} >0.
	$$
\end{theorem}

The proof of this result is rather involved, so we will present it divided into several steps. We start with the reduction to the case of $X$ and $Y$ being two-dimensional Banach spaces.

\begin{proposition} \label{BPBpp:dual4} Let $X$ and $Y$ be Banach spaces of dimension greater than or equal to $2$. Suppose that the pair $(X, Y)$ has property (P2). If $Y_0 \leq Y$ and $X_0 \leq X$ are such that $\dim(Y_0) = \dim (X/X_0) = 2$, then the pair $(X/X_0, Y_0)$ has property (P2).	
\end{proposition}

\begin{proof} Let $\eps > 0$ be given and assume that the pair $(X, Y)$ has property (P2) with some function $\bar\eta(\eps) > 0$. Let $\widetilde{T} : X/X_0 \longrightarrow Y_0$ with $\|\widetilde{T}\| = 1$ and $[x_0] \in S_{X/X_0}$ be such that
	\begin{equation*}
		\|\widetilde{T}([x_0])\| > 1 - \bar\eta(\eps/2).	
	\end{equation*}
	Pick a sequence $\{x_n\}_{n\in \N} \subset X$ with $\|x_n\| \longrightarrow 1$ and $[x_n] = [x_0]$ for every $n \in \N$. Consider the quotient mapping $Q: X \longrightarrow X/X_0$, define the operator $T:= \widetilde{T} \circ Q$, and observe that $\|T\|=1$ as $Q$ is a quotient map. Then
	\begin{equation*}
		\|T(x_n)\| = \|\widetilde{T}(Q(x_n))\| = \|\widetilde{T}([x_0])\| > 1 - \bar\eta(\eps/2).
	\end{equation*}	
	Therefore, we may find $n\in \N$ such that
	\begin{equation*}
		\left\| T\left(\frac{x_n}{\|x_n\|}  \right) \right\| > 1 - \bar\eta(\eps/2) \qquad \text{and} \qquad \bigl|1-\|x_n\|\bigr|<\eps/2.
	\end{equation*}
	The hypothesis provides us with $y_0 \in S_X$ such that
	\begin{equation*}
		\|T(y_0)\| = 1 \qquad \text{and} \qquad \left\| y_0 - \frac{x_n}{\|x_n\|} \right\| < \eps/2.
	\end{equation*}	
	Then, $\|x_n - y_0\| < \eps$ and so
	\begin{equation*}
		\|[x_0] - Q(y_0)\| \leq \|x_n - y_0\| < \eps.	
	\end{equation*}
	On the other hand,
	\begin{equation*}
		1 = \|T(y_0)\| = \| \widetilde{T}(Q(y_0))\| \leq \|Q(y_0)\| \leq 1	
	\end{equation*}
	which implies that $\|Q(y_0)\| = 1 = \|\widetilde{T}(Q(y_0))\|$.
\end{proof}

Therefore, the proof of Theorem \ref{theorem:dualpropertyisnotpossible} finishes if we are able to prove it for two-dimensional spaces $X$ and $Y$. This is what we will do in Proposition \ref{BPBpp:dual3}, but we need some preliminary work.

Let $X$ be a $2$-dimensional real Banach space. We assume that $X = \R^2$ and we consider the standard unit basis vectors $e_1 = (1, 0)$ and $e_2 = (0, 1)$ of $X$. The unit sphere $S_X$ of $X$ can be represented by the continuous curve $\gamma$ defined as follows:
\begin{equation*}
	\gamma (\theta) := \frac{\cos \theta e_1 + \sin \theta e_2}{\| \cos \theta e_1 + \sin \theta e_2\|} \qquad (\theta \in [0,2\pi]).	
\end{equation*}
Given a point $x = \gamma(\theta_0) \in S_X$ for some $\theta_0 \in \R$, we call the curve $\gamma_x$ defined by
\begin{equation*}
	\gamma_x(\theta) := \gamma (\theta + \theta_0) \qquad (0 \leq \theta \leq \pi)
\end{equation*}
as the \emph{half arc starting at} $x$. For $x^*\in S_{X^*}$, we define $F(x^*)$ to be the \emph{face} $F(x^*):=\{x\in S_X\colon x^*(x)=1\}$. We note that for a given $x^* \in S_{X^*}$, if $\gamma(\theta_1)$ and $\gamma(\theta_2)$ with $0 \leq \theta_2 - \theta_1 \leq \pi$ are in the face $F(x^*)$, then $\gamma(\theta) \in F(x^*)$ for all $\theta_1 \leq \theta \leq \theta_2$. Indeed, the line segment $[0, \gamma(\theta)]$ from $0$ to $\gamma(\theta)$ intersects the line segment $[\gamma(\theta_1), \gamma(\theta_2)]$ from $\gamma(\theta_1)$ to $\gamma(\theta_2)$ whenever $\theta_1 \leq \theta \leq \theta_2$ with $0 \leq \theta_2 - \theta_1 \leq \pi$ (see Figure \ref{fig1}). We will use this observation in the following result. More in general, we have for any $x^* \in X^*$ that if $x^*(\gamma(\theta_1)) \geq 1$ and $x^*(\gamma(\theta_2)) \geq 1$, then $x^*(\gamma(\theta)) \geq 1$ for all $\theta_1 \leq \theta \leq \theta_2$.

\begin{figure}[h]
	\includegraphics[width=10cm]{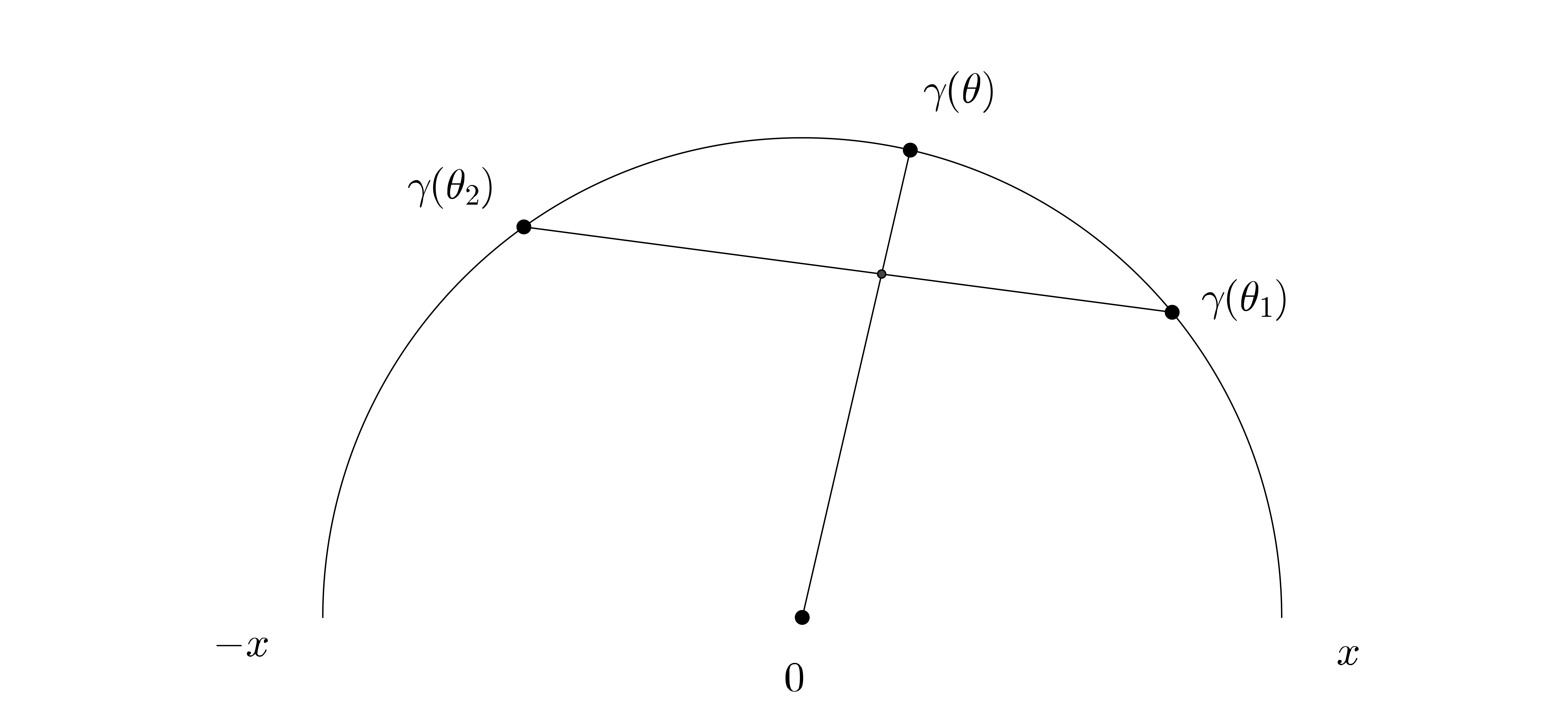}
	\centering
	\caption{The half arc starting at $x$}	
	\label{fig1}
\end{figure}

\begin{proposition} \label{BPBpp:dual1} Let $X$ and $Y$ be two-dimensional real Banach spaces and consider $T \in \mathcal{L}(X, Y)$ with $\|T\| = 1$. Let $\gamma$ be the half arc starting at $x$. Suppose that for $0 \leq \theta \leq \pi$, the image $T(\gamma(\theta))$ intersects the unit sphere $S_Y$ in three points at $0, \theta_c$ and $\pi$ for some $\theta_c$. Also, suppose that there are $\theta_1, \theta_2$ with $0 \leq \theta_1 \leq \theta_c$ and $\theta_c \leq \theta_2 \leq \pi$ such that $T(\gamma(\theta_1))$ and $T(\gamma(\theta_2))$ are in the interior of $B_Y$ (see Figure \ref{fig2}). Then $T(\gamma(\theta_c))$ does not belong to $F(y^*) \cup F(-y^*)$ for any $y^* \in S_{Y^*}$ with $y^*(T(x)) = 1$.	
\end{proposition}

\begin{figure}[h]
	\includegraphics[width=14cm]{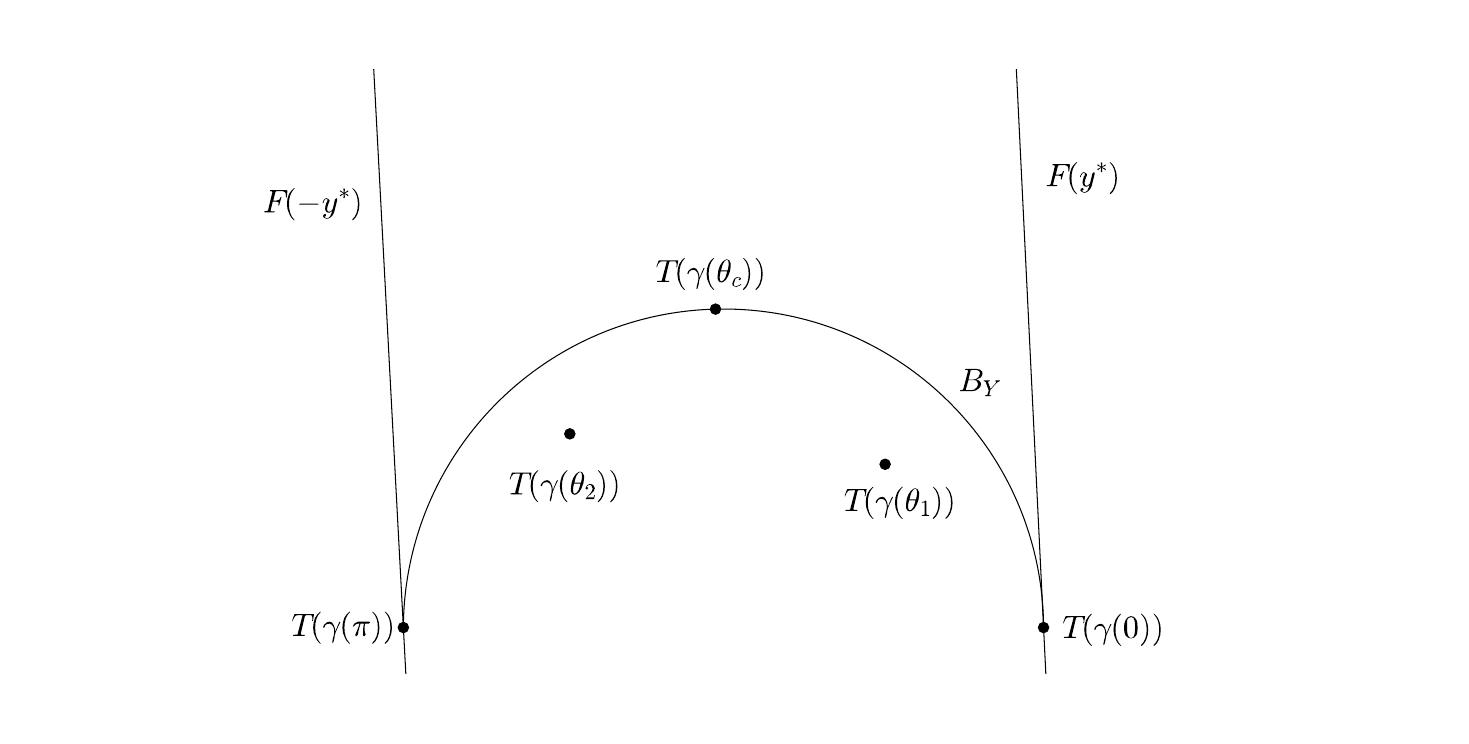}
	\centering
	\caption{ }	
	\label{fig2}
\end{figure}

\begin{proof} Suppose that there exists some $y^* \in S_{Y^*}$ such that $y^*(T(x)) = 1$ and $y^*(T(\gamma(\theta_c))) = 1$. Then
	\begin{equation*}
		[T^*y^*](x) = 1 = [T^*y^*](\gamma(\theta_c)).	
	\end{equation*}
	Note that $T^*y^* \in S_{X^*}$. So the points $x = \gamma(0)$ and $\gamma (\theta_c)$ are both in the face $F(T^*y^*)$. By the observation just before this proposition, we get that $\gamma(\theta) \in F(T^* y^*)$ for all $0 \leq \theta \leq \theta_c$. This implies that
	\begin{equation*}
		1 = [T^* y^*](\gamma(\theta)) = y^*(T(\gamma(\theta))) \leq \|T(\gamma(\theta))\| \leq 1	
	\end{equation*}	
	for all $0 \leq \theta \leq \theta_c$. This shows that $\|T(\gamma(\theta))\| = 1$ for all $0 \leq \theta \leq \theta_c$ which contradicts the hypothesis on $\theta_1$. If we have $T(\gamma(\theta_c)) \in F(-y^*)$, then we can use the same arguments as before to get a contradiction with the hypothesis on $\theta_2$.	
\end{proof}

The most intriguing part of the proof of Theorem \ref{theorem:dualpropertyisnotpossible} for a pair of two-dimensional real spaces is contained in the following proposition which may have its own interest.

\begin{proposition} \label{BPBpp:dual2} Let $X$ and $Y$ be $2$-dimensional  real Banach spaces. Then there exists $T \in \mathcal{L}(X, Y)$ with $\|T\| = 1$ such that
	\begin{itemize}
		\item[(i)] $T(B_X) \subset B_Y$ and
		\item[(ii)] $T(B_X) \cap S_Y$ contains two points $y_1$ and $y_2$ such that for some $y_1^* \in S_{Y^*}$ with $y_1^*(y_1) = 1$ we have $\dist (y_2, F(y_1^*) \cup F(-y_1^*)) > 0$.
	\end{itemize}	
\end{proposition}

\begin{proof} We divide the proof in two cases.

\noindent	{\it Case 1:} we assume that $X$ is a Hilbert space. Since $Y$ is finite-dimensional, by using John's theorem (see \cite[Corollary 15.2, p. 121]{T} for example), there is a unique ellipsoid $\mathcal{E}$ of maximal volume such that $\mathcal{E} \subset B_Y$. Since $X$ is a Hilbert space, there is $T \in \mathcal{L}(X, Y)$ with $\|T\| = 1$ such that $T(B_X) = \mathcal{E} \subset B_Y$. Now using \cite[Theorem 15.3]{T}, since $Y$ is $2$-dimensional, there are at least two linearly independent points $y_1, y_2 \in T(B_X) \cap S_Y$. Let $y_1^* \in S_{Y^*}$ be such that $y_1^*(y_1) = 1$. Since the boundary of $\mathcal{E}$ does not contain line segments, we get that $y_2 \not\in F(y_1^*) \cup F(y_2^*)$.

	\begin{figure}[h]
		\includegraphics[width=17cm]{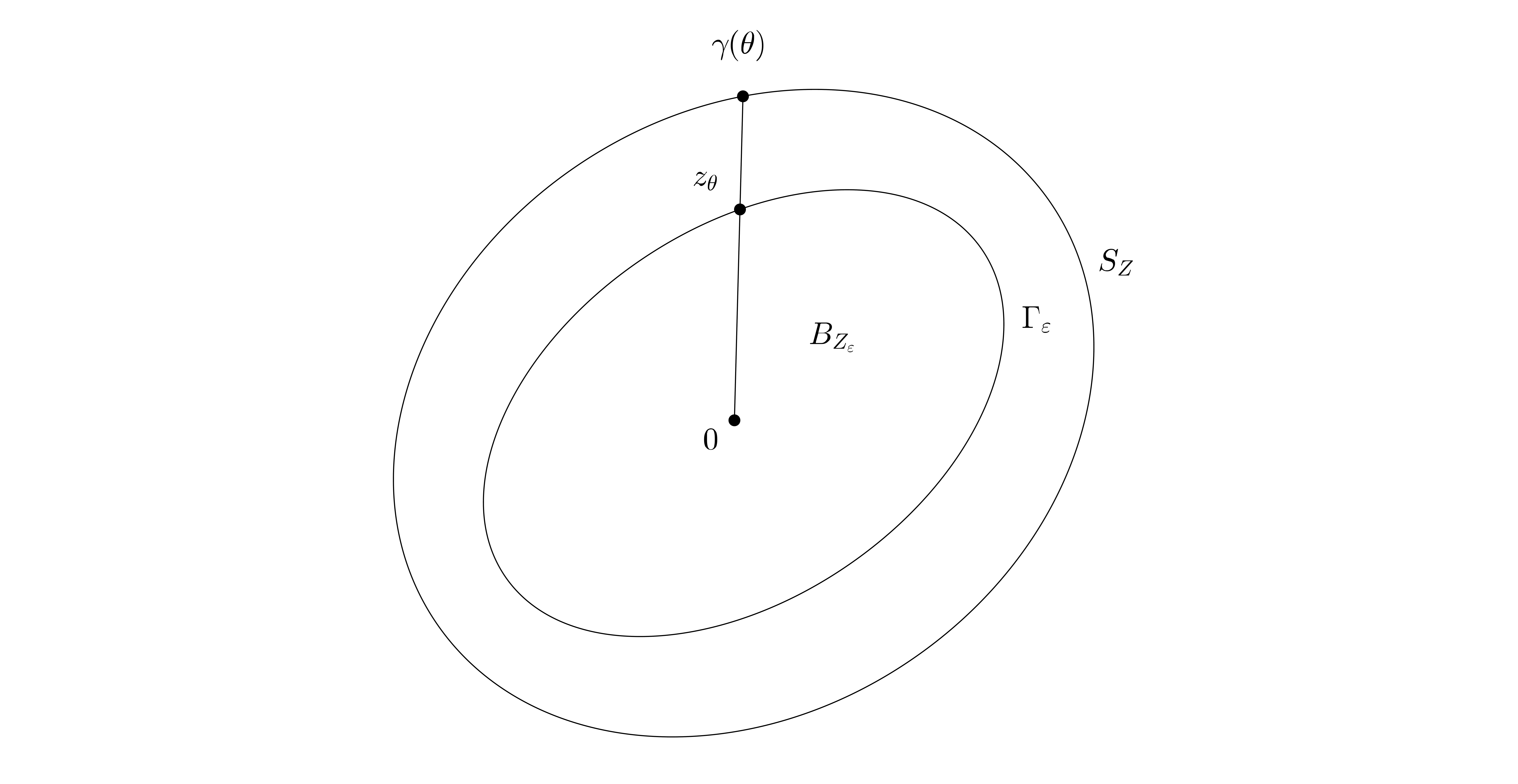}
		\centering
		\caption{ }	
		\label{fig5}
	\end{figure}

	Before consider Case 2 in which $X$ is not a Hilbert space, we review the proof of \cite[Theorem]{N}. Let $Z$ be any $2$-dimensional Banach space and let $\gamma$ be a parametrization of $S_Z$. If two unit vectors $z_1$ and $z_2$ are rotated around $S_Z$ while their difference $z_1 - z_2$ has constantly norm equal to $\eps$, the vector $\frac{1}{2} (z_1 + z_2)$ describes a curve $\Gamma_{\eps}$. Let $r(\theta) = \| \gamma(\theta)\|_2$ where $\| \cdot \|_2$ denotes the Euclidean norm in $\R^2$. Let $z_{\theta}$ be the point where the segment $[0, \gamma(\theta)]$ intersects $\Gamma_{\eps}$ (see Figure \ref{fig5}) and let $\Delta(\eps, \theta) = 1 - \|z_{\theta}\| = \| \gamma(\theta) - z_{\theta}\|$. So $\|z_{\theta}\| = 1 - \Delta(\eps, \theta)$ and $z_{\theta} = \|z_{\theta}\| \gamma(\theta)$. So
	\begin{equation*}
		\|z_{\theta}\|_2 = \|z_{\theta}\| \|\gamma(\theta)\|_2 = (1 - \Delta(\eps, \theta))\|\gamma(\theta)\|_2 = (1 - \Delta(\eps, \theta))r(\theta).	
	\end{equation*}
	Using this and denoting $B_{Z_{\eps}}$  the region inside $\Gamma_{\eps}$, we have that
	\begin{equation*}
		\Area (B_{Z_{\eps}}) = \frac{1}{2} \int_0^{2 \pi} \|z_{\theta}\|_2^2 d \theta = \frac{1}{2} \int_0^{2\pi} (1 - \Delta(\eps, \theta))^2 r(\theta)^2 d \theta.
	\end{equation*}
	Also,
	\begin{equation*}
		\Area(B_Z) = \frac{1}{2} \int_0^{2\pi} \|\gamma(\theta)\|_2^2 d \theta = \frac{1}{2} \int_{0}^{2\pi} r(\theta)^2 d \theta.	
	\end{equation*}
	On the other hand, \cite[Lemma]{N} says that
	\begin{equation*}
		\Area (B_{Z_{\eps}}) = \left(1 - \frac{\eps^2}{4} \right) \Area (B_Z).	
	\end{equation*}
	And then
	\begin{equation} \label{nord1}
		\int_0^{2 \pi} \left[ \left( 1 - \Delta(\eps, \theta) \right)^2 - \left(1 - \frac{\eps^2}{4}\right)  \right]r(\theta)^2 d \theta = 0.
	\end{equation}

\noindent	{\it Case 2:} Now we assume that $X$ is not a Hilbert space. By the Day-Nordlander theorems (see \cite[p. 60]{Diestel} or \cite[Theorem 4.1]{Day} and \cite[Theorem]{N}, respectively), there is some $\eps > 0$ such that $\delta_X(\eps)$ is strictly less than the modulus of convexity of a Hilbert space $\delta_H(\eps) = 1 - \sqrt{1 - \frac{\eps^2}{4}}$.
	So by (\ref{nord1}), there is $\theta_0$ such that
	\begin{equation*}
		(1 - \Delta(\eps, \theta_0))^2 - \left(1 - \frac{\eps^2}{4} \right) < 0.	
	\end{equation*}
	as well as $\theta_1$ such that
	\begin{equation*}
		(1 - \Delta(\eps, \theta_1))^2 - \left(1 - \frac{\eps^2}{4} \right) > 0.	
	\end{equation*}
	It means that there are $x_1, x_2 \in S_X$ such that  $\|x_1 - x_2\| = \eps$ and $\|x_1 + x_2\| < \sqrt{4 - \eps^2}$. By moving one of the points $x_1$ or $x_2$ on $S_X$ a little, we may assume that those points satisfy $\|x_1 + x_2\| < \sqrt{4 - \eps^2}$ and $\|x_1 - x_2\| < \eps$.

	Now for the Banach space $Y$, using the continuity of $\Delta$, we can find $\theta_2$ such that
	\begin{equation*}
		(1 - \Delta(\eps, \theta_2))^2 - \left( 1 - \frac{\eps^2}{4} \right) = 0.	
	\end{equation*}
	So there are $y_1, y_2 \in S_Y$ such that $\|y_1 - y_2\| = \eps$ and $\|y_1 + y_2\| = \sqrt{4 - \eps^2}$.

	Define the operator $S: X \longrightarrow Y$ to be such that
	\begin{equation*}
		S(x_1) = y_1 \qquad \text{and} \qquad S(x_2) = y_2.	
	\end{equation*}
	So $\|S(x_1)\| = \|S(x_2)\| = 1$, $S(x_1 - x_2) = y_1 - y_2$ and $S(x_1 + x_2) = y_1 + y_2$. Moreover,
	\begin{equation*}
		\left\| S \left( \frac{x_1 - x_2}{\|x_1 - x_2\|}   \right)	\right\| = \frac{\|y_1 - y_2\|}{\|x_1 - x_2\|} > 1 \quad \text{and} \quad \left\| S \left( \frac{x_1 + x_2}{\|x_1 + x_2\|}   \right)	\right\| = \frac{\|y_1 + y_2\|}{\|x_1 + x_2\|} > 1
	\end{equation*}
	Multiplying the operator $S$ by $1 - \delta$ for some small $\delta > 0$, we may assume that
	\begin{equation*}
		\|S(x_1)\| < 1, \quad \|S(x_2)\| < 1, \quad \left\| S \left( \frac{x_1 - x_2}{\|x_1 - x_2\|}   \right)	\right\| > 1 \quad \text{and}\quad  \left\| S \left( \frac{x_1 + x_2}{\|x_1 + x_2\|}   \right)	\right\| > 1.
	\end{equation*}
	Consider $\gamma_1$ to be the half arc starting at $\frac{x_1 + x_2}{\|x_1 + x_2\|}$ (see Figure \ref{fig3}). Then there are $0 \leq t_1 < t_2 < t_3 \leq \pi$ such that
	\begin{equation*}
		\|S(\gamma_1(t_1))\| < 1, \qquad \|S(\gamma_1(t_2))\| > 1 \qquad \text{and} \qquad \|S(\gamma_1(t_3))\| < 1.	
	\end{equation*}
	\begin{figure}[t]	
		\includegraphics[width=14cm]{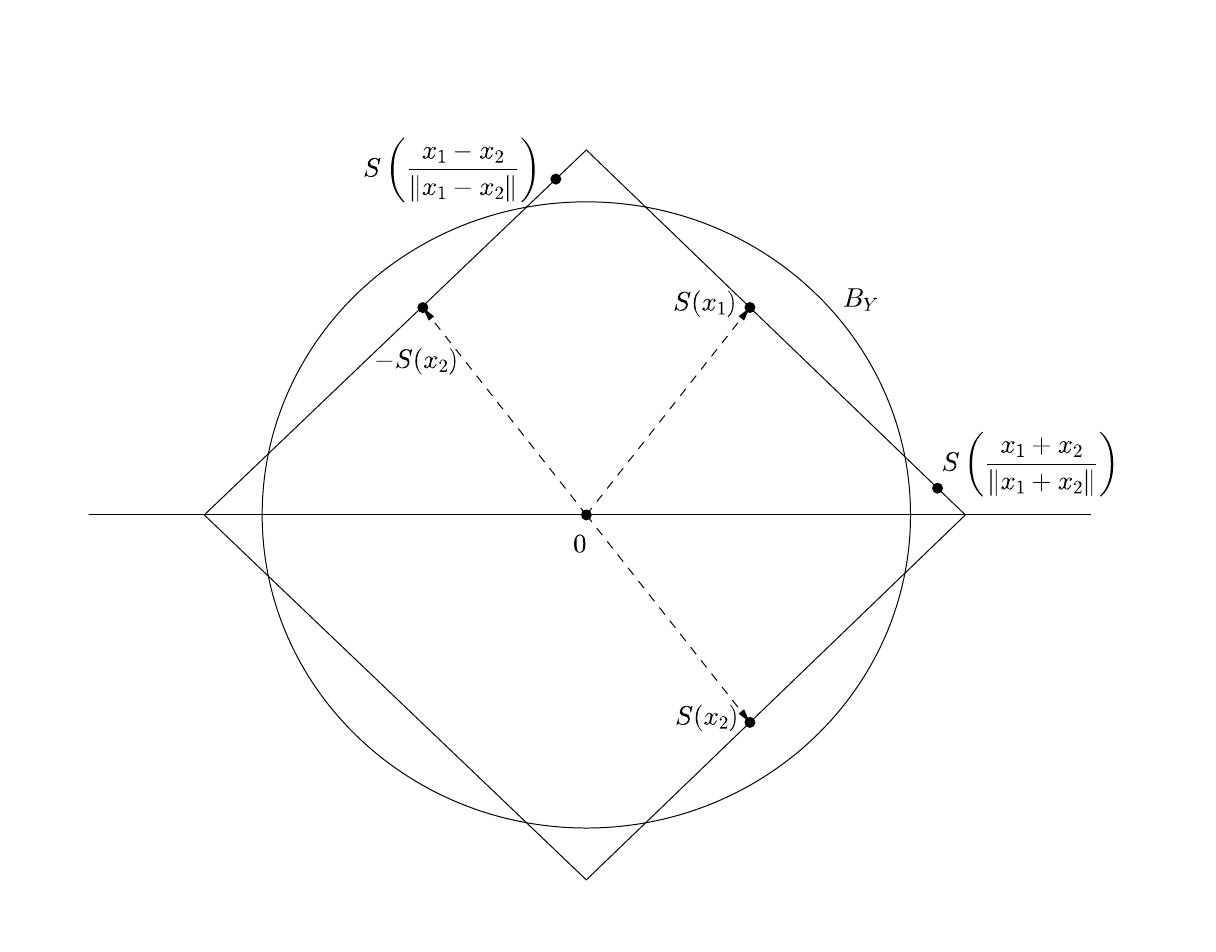}
		\centering
		\caption{ }	
		\label{fig3}
	\end{figure}
	Let
	\begin{equation*}
		a := \max \bigl\{\|S(\gamma_1(t))\|\colon 0 \leq t \leq t_1, \ t_3 \leq t \leq \pi \bigr\} \quad \text{and} \quad b := \max\bigl\{\|S(\gamma_1(t))\|\colon \ t_1 \leq t \leq t_3 \bigr\}.	
	\end{equation*}
	
	We may assume that $a\leq b$. Otherwise, we consider the half arc $\gamma_2$ starting at $\frac{x_1 - x_2}{\|x_1 - x_2\|}$ instead of $\gamma_1$. Now we consider two cases.
	
\noindent	{\it Subcase 1:} We assume that $a = b$ and we consider the operator $T := \frac{1}{a}S \in S_{\mathcal{L}(X, Y)}$. So
	\begin{equation*}
		\|T(\gamma_1(t_1))\| = \frac{1}{a}\|S(\gamma_1(t_1))\| \leq \|S(\gamma_1(t_1))\| < 1.	
	\end{equation*}
	Analogously, $\|T(\gamma_1(t_3))\| < 1$. Also, by the definition of $a$ and $b$, there are $s_1, s_2$ such that $s_1 \leq t_1 < s_2 < t_3 \leq s_1 + \pi$ such that
	\begin{equation*}
		\|T(\gamma_1(s_1))\| = \|T(\gamma_1(s_1+ \pi))\| = \|T(\gamma_1(s_2))\| = 1.
	\end{equation*}
	Let $y_1^* \in S_{Y^*}$ be such that $y^*(T(\gamma_1(s_1))) = 1$. Define $y_1 := T(\gamma_1(s_1))$ and $y_2 := T(\gamma_1(s_2))$. So $y_1^*(y_1) = 1$ and by Proposition \ref{BPBpp:dual1}, $\dist\left( y_2, F(y_1^*) \cup F(-y_1^*) \right) > 0$.

\noindent	{\it Subcase 2:} Now we assume that $a < b$. Let $s \in [0, t_1] \cup [t_3, \pi]$ be such that $a = \|S(\gamma_1(s))\|$. Define the operator $T_1 := \frac{1}{a} S$. Then $\|T_1(\gamma_1(s))\| = \frac{1}{a}\|S(\gamma_1(s))\| = 1$. Let $\gamma_2$ be the half arc starting at $\gamma_1(s)$. So (see Figure \ref{fig4}) there are $0 < s_1 < s_2 < s_3 < \pi$ such that
	\begin{equation*}
		\|T_1(\gamma_2(0))\| = \|T_1(\gamma_2(\pi))\| = 1,\quad \|T_1(\gamma_2(t))\| \leq 1 \qquad \text{for} \ t \in [0, s_1] \cup [s_3, \pi]
	\end{equation*}
	as well as
	\begin{equation*}
		\|T_1(\gamma_2(s_1))\| <1, \quad \|T_1(\gamma_2(s_3))\| < 1 \quad \text{and} \quad \|T_1(\gamma_2(s_2))\| > 1.	
	\end{equation*}
	\begin{figure}[h]
		\includegraphics[width=17cm]{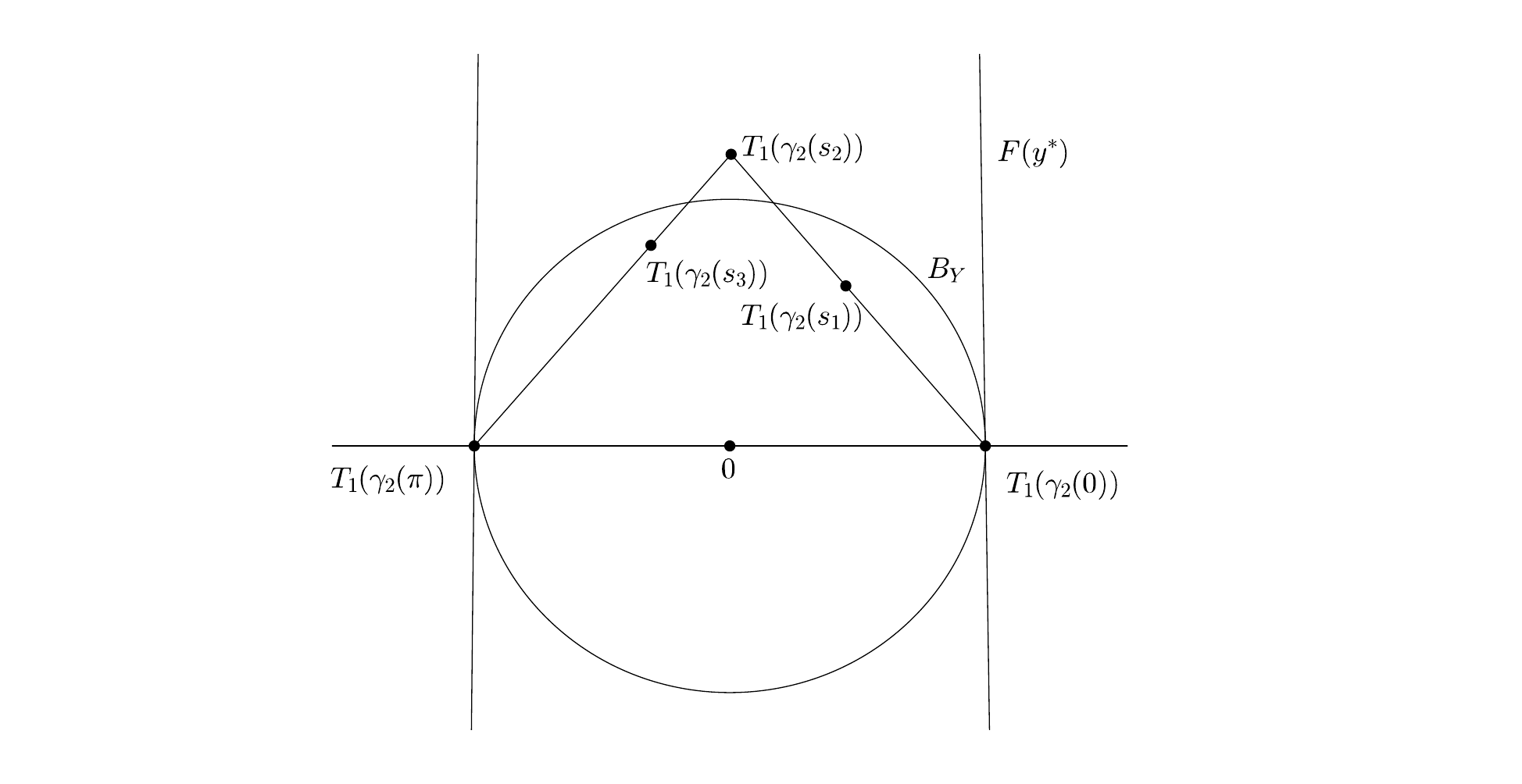}
		\centering
		\caption{ }	
		\label{fig4}
	\end{figure}
	
	Let $y^* \in S_{Y^*}$ be such that $y^*(T_1(\gamma_2(0))) = 1$ and define $P: Y \longrightarrow Y$ by
	\begin{equation*}
		P(y) := y^*(y)T_1(\gamma_2(0)) \qquad (y \in Y).
	\end{equation*}
	Note that $P$ is a projection with $\|P\| = 1$. For all $\lambda \in [0, 1]$, we define $P_{\lambda}: Y \longrightarrow Y$ by
	\begin{equation*}
		P_{\lambda} := \lambda \Id_Y + (1 - \lambda)P.	
	\end{equation*}
	So $\|P_{\lambda}\| \leq 1$. Let $T_{\lambda} := P_{\lambda}T_1 \in \mathcal{L}(X, Y)$ and define $\varphi: [0, 1] \longrightarrow \R$ by
	\begin{equation*}
		\varphi(\lambda) := \max \bigl\{\|P_{\lambda} T_1(\gamma_2(t))\|\colon \ s_1 \leq t \leq s_3 \bigr\} \qquad (\lambda \in [0, 1]).
	\end{equation*}
	Then $\varphi$ is continuous,
	\begin{equation*}
		\varphi(0) = \max \bigl\{|y^*(T_1(\gamma_2(t)))|\colon \ s_1 \leq t \leq s_3 \bigr\} \qquad \text{and} \qquad \varphi(1) = \max \bigl\{\|T_1(\gamma_2(t))\|\colon \ s_1 \leq t \leq s_3 \bigr\}.
	\end{equation*}
	We note that $|y^*(T_1(\gamma_2(t)))| < 1$ for all $s_1 \leq t \leq s_3$. Indeed, otherwise there is some $s_1 < \widetilde{t} < s_3$ such that $y^*(T_1(\gamma_2(\widetilde{t}))) = 1$ or $-1$. We assume that $y^*(T_1(\gamma_2(\widetilde{t}))) = 1$, so
	\begin{equation*}
		[T_1^*y^*](\gamma_2(\widetilde{t})) = 1 = [T_1^*y^*](\gamma_2(0)).	
	\end{equation*}
	Hence $[T_1^*y^*](\gamma_2(t)) \geq 1$ for all $0 \leq t \leq \widetilde{t}$ and this is a contradiction with the fact that $$y^*(T_1(\gamma_2(s_1))) \leq \|T_1(\gamma_2(s_1))\| < 1.$$ Therefore, $\varphi(0) < 1$. Since $\varphi(1) \geq \|T_1(\gamma_2(s_2))\| > 1$, there exists $\lambda_0 \in (0, 1)$ such that $\varphi(\lambda_0) = 1$. Consider $T := T_{\lambda_0} = P_{\lambda_0} T_1 \in \mathcal{L}(X, Y)$. Then $\|T\| \leq 1$. Also,
	\begin{equation*}
		\|T(\gamma_2(0))\| = \|T(\gamma_2(\pi))\| = 1, \quad \|T(\gamma_2(s_1))\| < 1 \quad \text{and} \quad \|T(\gamma_2(s_3))\| < 1.	
	\end{equation*}
	Also, by the definition of $\varphi(\lambda_0)$, there is $\widetilde{s}_2 \in [s_1, s_3]$ such that $\|T(\gamma_2(\widetilde{s}_2))\| = 1$. So taking $y_1 := T(\gamma_2(0))$, $y_2 := T(\gamma_2(\widetilde{s}_2))$ and $y_1^* \in S_{Y^*}$ to be such that $y_1^*(y_1) = 1$, one has $\dist\left( y_2, F(y_1^*) \cup F(-y_1^*) \right) > 0$ by Proposition \ref{BPBpp:dual1}, as desired.
\end{proof}

We are now ready to prove that a pair $(X,Y)$ with $\dim(X)=\dim(Y)=2$ cannot satisfy property (P2). As announced, this, together with Proposition \ref{BPBpp:dual4}, provide the proof of Theorem \ref{theorem:dualpropertyisnotpossible}.

\begin{proposition} \label{BPBpp:dual3} Let $X$ and $Y$ be $2$-dimensional real Banach spaces. Then there are $\delta > 0$, $T_n \in \mathcal{L}(X, Y)$ with $\|T_n\| = 1$ for every $n\in \N$, and $x_0 \in S_X$, such that
	\begin{equation*}
		\|T_n(x_0)\| \longrightarrow 1	
	\end{equation*}
	but $\dist \left(x_0, \bigl\{ x \in S_X\colon \|T_n(x)\| = 1 \bigr\}\right) > \delta$ for every $n \in \N$.
\end{proposition}

\begin{proof} By Proposition \ref{BPBpp:dual2}, there exists an operator $T \in \mathcal{L}(X, Y)$ with $\|T\| = 1$ so that $T(B_X) \cap S_Y$ contains two points $y_1$ and $y_2$ in such a way that for some $y_1^* \in S_{Y^*}$ with $y_1^*(y_1) = 1$ we have $$\delta:= \dist (y_2, F(y_1^*) \cup F(-y_1^*)) > 0.$$ Let $P: Y \longrightarrow Y$ be the projection defined by
	\begin{equation*}
		P(y) := y_1^*(y)y_1 \qquad  (y \in Y).
	\end{equation*}
	For all $\lambda \in [0, 1]$, define
	\begin{equation*}
		P_{\lambda} := \lambda \Id_Y + (1 - \lambda) P \in B_{\mathcal{L}(Y, Y)}.
	\end{equation*}
	Since $\|T\|=1$ and $y_2 \in T(B_X)\cap S_Y$, there exists $x_0 \in S_X$ such that $T(x_0) = y_2$.  Note that if $y \in B_Y \setminus (F(y_1^*) \cup F(-y_1^*))$, then $\|P(y)\| < 1$ and $P_{\lambda}(y)$ is in the interior of $B_Y$ for all $0 \leq \lambda < 1$. Therefore,
	\begin{equation*}
		\lim_{\lambda \rightarrow 1} \|P_{\lambda}(T(x_0))\| = \lim_{\lambda \rightarrow 1} \|P_{\lambda} (y_2)\| = \lim_{\lambda \rightarrow 1} \|\lambda y_2 + (1 - \lambda)P(y_2)\| = 1.	
	\end{equation*}
	Let $T_{\lambda} := P_{\lambda} T \in \mathcal{L}(X, Y)$. If $x \in S_X$ is such that $\|T_{\lambda} (x) \| = 1$, we have that
	\begin{equation*}
		1 = \|T_{\lambda}(x)\|  = \|P_{\lambda}(T(x))\| \leq \lambda \|T(x)\| + (1 - \lambda) |y_1^*(T(x))|\|y_1\| \leq 1	
	\end{equation*}
	which implies that $|y_1^*(T(x))| = 1$ and so $T(x) \in F(y_1^*) \cup F(-y_1^*)$. So
	\begin{equation*}
		\|x - x_0\| \geq \|T(x) - T(x_0)\| = \|T(x) - y_2\| > \delta.	\qedhere
	\end{equation*}
\end{proof}

\end{document}